\documentclass{amsart}
\usepackage{amssymb}
\newtheorem{theorem}{Theorem}[section]
\newtheorem{lemma}[theorem]{Lemma}

\newtheorem{remark}[theorem]{Remark}
\newtheorem{definition}[theorem]{Definition}

\newtheorem{corollary}[theorem]{Corollary}
\newtheorem{proposition}[theorem]{Proposition}
\newcommand{\R}{\mathbb R}
\newcommand{\N}{\mathbb N}
\newcommand{\Z}{\mathbb Z}
\newcommand{\Q}{\mathbb Q}

\newcommand{\F}{\mathbb F}
\newcommand{\st}[1]{\vskip 1mm\noindent{\bf #1}\ }
\newcommand{\stst}[1]{\vskip 1mm\noindent\hskip5mm{\bf #1}\ }
\newcommand{\ststst}[1]{\vskip 1mm\noindent\hskip10mm{\bf #1}\ }
\newcommand{\stststst}[1]{\vskip 1mm\noindent\hskip15mm{\bf #1}\ }
\newcommand{\ststststst}[1]{\vskip 1mm\noindent\hskip20mm{\bf #1}\ }
\def\op{\operatorname}

\def\as#1{\renewcommand\arraystretch{#1}}

\def\chr{\op{char}}

\def\cs{\operatorname{cs}}

\def\df{\operatorname{Diff}}

\def\dsc{\operatorname{Disc}}
\def\diso{\lower.4ex\hbox{$\downarrow$}\raise.4ex\hbox{\mbox{\scriptsize
$\wr$}}}

\def\exp{\op{exp}}

\def\ind{\op{ind}}
\def\iso{\ \lower.3ex\hbox{\as{.08}$\begin{array}{c}\lra\\\mbox{\tiny $\sim\,$}\end{array}$}\ }

\def\kb{\overline{k}}

\def\ks{k^{\op{sep}}}

\def\la{\lambda}

\def\lg{l\raise.6ex\hbox to.2em{\hss.\hss}l}
\def\lra{\longrightarrow}
\def\m{{\mathfrak m}}
\def\md#1{\ \mbox{\rm(mod }{#1})}

\def\om{\omega}
\def\oo{{\mathcal O}}
\def\orb{\hbox to  .3em{$\backslash$}\backslash}
\def\ord{\op{ord}}
\def\p{\mathfrak{p}}
\def\P{\mathfrak{P}}

\def\q{\mathfrak{Q}}

\def\res{\operatorname{Res}}

\def\t{\theta}

\def\ty{\mathbf{t}}

\newcounter{cs}
\stepcounter{cs}
\newcommand{\casos}{\begin{itemize}}
\newcommand{\fcasos}{\end{itemize}\setcounter{cs}{1}}

\newfont{\tit}{cmr12 scaled \magstep3}

\begin{document}
\title{Local computation of differents and discriminants}

\author[Nart]{Enric Nart}
\address{Departament de Matem\`{a}tiques,
         Universitat Aut\`{o}noma de Barcelona,
         Edifici C, E-08193 Bellaterra, Barcelona, Catalonia, Spain}
\email{nart@mat.uab.cat}
\thanks{Partially supported by MTM2009-10359 from the
Spanish MEC}
\date{}
\keywords{different, discriminant, global field, local field, Montes algorithm, Newton polygon, Okutsu invariant, OM representation, resultant, Single-factor lifting algorithm}

\makeatletter
\@namedef{subjclassname@2010}{%
  \textup{2010} Mathematics Subject Classification}

\subjclass[2010]{Primary 11Y40; Secondary 11Y05, 11R04, 11R27}

\begin{abstract}
We obtain several results on the computation of different and discriminant ideals of finite extensions of local fields.  
As an application, we deduce routines to compute the $\p$-adic valuation of the discriminant $\dsc(f)$, and the resultant $\res(f,g)$, for polynomials $f(x),g(x)\in A[x]$, where $A$ is a Dedekind domain and $\p$ is a non-zero prime ideal of $A$ with finite residue field. These routines do not require the computation of neither $\dsc(f)$ nor $\res(f,g)$; hence, they are useful in cases where this latter computation is inefficient because the polynomials have a large degree or very large coefficients. 
\end{abstract}

\maketitle
\section*{Introduction}
Let $A$ be a Dedekind domain whose field of fractions $K$ is a global field. Typical instances of $A$ are the ring of integers of a number field, or the polynomial ring in one indeterminate over a finite field. Let $L/K$ be a finite separable extension and $B$ the integral closure of $A$ in $L$. Let $\t\in L$ be a primitive element of $L/K$, with minimal polynomial $f(x)\in A[x]$. Also, let $\p$ be a non-zero prime ideal of $A$, $v_\p$ the canonical $\p$-adic valuation and $K_\p$ the completion of $K$ at $\p$.

The Montes algorithm \cite{algorithm} computes an \emph{OM representation} of every prime ideal $\P$ of $B$ lying over $\p$ \cite{newapp}. An OM representation is a computational object supporting several data and operators, linked to one of the irreducible factors (say) $F(x)$ of $f(x)$ in $K_\p[x]$. Among these data, the OM representation contains the \emph{Okutsu invariants} of $F$, which reveal a lot of arithmetic information about the finite extension of $K_\p$ determined by $F$ \cite{Ok,okutsu}. 

In this paper, we present an algorithm to compute the $\P$-adic valuation of the different ideal $\df(L/K)$ in terms of the OM representation of $\P$. Of course, this determines the $\p$-adic valuation of the discriminant $\dsc(L/K)$. This computation does not require the computation of the discriminant $\dsc(f)$ of $f(x)$; actually, the algorithm provides the value of $v_\p(\dsc(f))$ as a by-product. We use this fact to derive a routine to compute $v_\p(\dsc(g))$ for an arbitrary polynomial $g(x)\in K[x]$. Also, we use similar ideas to compute the $\p$-adic valuation of the resultant $\res(g,h)$ of two polynomials $g(x),h(x)\in K[x]$. These routines may be useful in cases where the natural computation of $\dsc(g)$ or $\res(g,h)$ is inefficient because the polynomials have a large degree or very large coefficients. 

The algorithms are described in sections 2, 3. If $A/\p$ is small, the computation of $\delta:=v_\p(\dsc(g))$ requires $O\left(n^{2+\epsilon}+n^{1+\epsilon}\delta^{2+\epsilon}\right)$ word operations, where $n=\deg g$.
The computation of $v_\p(\res(g,h))$ has the same complexity, but taking $n=\max\{\deg g,\deg h\}$, $\delta=v_\p(\res(g,h))$. The algorithms have an excellent practical performance. In section 4 we present some numerical tests with Magma, in which we compare running times with the naive algorithms that compute first $\dsc(g)$ or $\res(g,h)$, and then its $\p$-valuation.  

The core of the paper is section 1, where we discuss the computation of the different ideal of an extension of local fields. Let $L_\P$ be the completion of $L$ at the prime ideal $\P$. Consider an embedding, $L\subset L_\P\hookrightarrow \overline{K}_\p$, of $L$ into a fixed algebraic closure of $K_\p$, and let $F(x)\in K_\p[x]$ be the monic irreducible factor of $f(x)$ that vanishes on the image of $\t$ in $\overline{K}_\p$, which we denote still by the same symbol $\t$. Let $e(F)=e(\P/\p)$ be the ramification index of $\P$ over $\p$, and  $f(F)=f(\P/\p)$ the residual degree. The different ideal $\df(L_\P/K_\p)$ is equal to the $(e(F)-1+\rho)$-th power of the maximal ideal of the ring of integers of $L_\P$, where $\rho$ is a non-negative integer that vanishes if and only if $\P$ is tamely ramified over $\p$. On the other hand, from Theorem \ref{vdisc} we deduce identities:
$$
v_\p(\dsc(F))=nv_\p(F'(\t))=n\mu(F)+f(F)\rho,
$$
for some invariant $\mu(F)\in\Q$. The three numbers $e(F)$, $f(F)$, $\mu(F)$ are Okutsu invariants of $F$, which can be expressed in terms of data contained in the OM representation of $F$ by a direct formula (Proposition \ref{okutsu}).
Therefore, for the computation of $\df(L_\P/K_\p)$ and $v_\p(\dsc(F))$, we need only to compute $v_\p(F'(\t))$.    
Since it is impossible in general to get an exact computation of the polynomial $F(x)$, in Theorem \ref{howmuch} we show how to construct a polynomial $\phi(x)\in A[x]$ which is sufficiently close to $F(x)$ to have $v_\p(\phi'(\t))=v_\p(F'(\t))$. Finally, once we get this sufficiently good approximation $\phi$ to $F$, the value of $v_\p(\phi'(\t))$ may be determined by using Newton polygons of higher order (Proposition \ref{vgt}), whose computation relies as well in data contained in the OM representation. 

\section{Computation of local different ideals}\label{secDiff}
Let $k$ be a local field, i.e. a locally compact and complete field with respect to a discrete valuation $v$. Let $\oo$ be the ring of integers of $k$, $\m$ the maximal ideal, and $\F=\oo/\m$ the residue field, which is a finite field. Let $p$ be the characteristic of $\F$. 

Let $\ks\subset \kb$ be the separable closure of $k$ inside a fixed algebraic closure. Let $v\colon \kb\to \Q\cup\{\infty\}$, be the canonical extension of the discrete valuation $v$ to $\kb$, normalized by $v(k)=\Z$.

\subsection{Okutsu invariants of an irreducible separable polynomial}\label{subsecOkutsu} Let $F(x)\in\oo[x]$ be a monic irreducible separable polynomial, $\t\in \ks$ a root of $F(x)$, and $L=k(\t)$ the finite separable extension of $k$ generated by $\t$. Denote $n:=[L\colon k]=\deg F$. Let $\oo_L$ be the ring of integers of $L$, $\m_L$ the maximal ideal and $\F_L$ the residue field. 
We indicate with a bar, $\raise.8ex\hbox{---}\colon \oo[x]\longrightarrow \F[x]$,
the canonical homomorphism of reduction of polynomials modulo $\m$. 

Let $[\phi_1,\dots,\phi_r]$ be an \emph{Okutsu frame} of $F(x)$, and let $\phi_{r+1}$ be an \emph{Okutsu approxi\-mation} to $F(x)$. That is, $\phi_1,\dots,\phi_{r+1}\in\oo[x]$ are monic separable polynomials of strictly increasing degree: $$1\le m_1:=\deg\phi_1<\cdots <m_r:=\deg\phi_r<m_{r+1}:=\deg\phi_{r+1}=n,$$ and for any monic polynomial $g(x)\in\oo[x]$ we have, for all $0\le i\le r$: 
$$
m_i\le \deg g<m_{i+1}\ \Longrightarrow\ \dfrac{v(g(\t))}{\deg g}\le\dfrac{v(\phi_i(\t))}{m_i}<\dfrac{v(\phi_{i+1}(\t))}{m_{i+1}},
$$
with the convention that $m_0=1$ and $\phi_0(x)=1$. It is easy to deduce from these conditions that the polynomials $\phi_1(x),\dots,\phi_{r+1}(x)$ are all irreducible in $\oo[x]$. 

The length $r$ of the frame is called the \emph{Okutsu depth} of $F(x)$. 
Okutsu frames were introduced by K. Okutsu in \cite{Ok} as a tool to construct integral bases. Okutsu approximations were introduced in \cite{okutsu}, where it is shown that the family $\phi_1,\dots,\phi_{r+1}$ determines an \emph{optimal $F$-complete type of order $r+1$}:
\begin{equation}\label{OM}
\ty_F=
(\psi_0;(\phi_1,\lambda_1,\psi_1);\cdots;(\phi_r,\lambda_r,\psi_r);(\phi_{r+1},\lambda_{r+1},\psi_{r+1})).
\end{equation}
In the special case $\phi_{r+1}=F$, we have $\la_{r+1}=-\infty$ and $\psi_{r+1}$ is not defined.  We call $\ty_F$ an \emph{OM representation} of $F$.

An OM representation of the polynomial $F$ carries (stores) several invariants and operators yielding strong arithmetic information about $F$ and the extension $L/k$.
Let us recall some of these invariants and operators.

Attached to the type $\ty_F$, there is a family of discrete valuations of the rational function field $k(x)$, the \emph{MacLane valuations}:
$$
v_i\colon k(x)\longrightarrow \Z\cup\{\infty\},\quad 1\le i\le r+1,
$$
such that $0=v_1(F)<\cdots<v_{r+1}(F)$. The $v_1$-value of a polynomial in $k[x]$ is the minimum of the $v$-values of its coefficients.

Also, $\ty_F$ determines a family of Newton polygon operators:
$$
N_i\colon k[x]\longrightarrow 2^{\R^2},\quad 1\le i\le r+1,
$$where $2^{\R^2}$ is the set of subsets of the Euclidean plane. Any non-zero polynomial $g(x)\in k[x]$ has a canonical $\phi_i$-development:
$$
g(x)=\sum\nolimits_{0\le s}a_s(x)\phi_i(x)^s,\quad \deg a_s<m_i,
$$ and the polygon $N_i(g)$ is the lower convex hull of the set of points $(s,v_i(a_s\phi_i^s))$. Usually, we are only interested in the principal polygon $N_i^-(g)\subset N_i(g)$ formed by the sides of negative slope. For all $1\le i\le r$, the Newton polygons $N_i(F)$ and $N_i(\phi_{i+1})$ are one-sided and they have the same slope, which is a negative rational number $\lambda_i\in\Q_{<0}$.  
The Newton polygon $N_{r+1}(F)$ is one-sided and it has an (extended) integer negative slope, which we denote by $\lambda_{r+1}\in\Z_{<0}\cup\{-\infty\}$.

There is a chain of finite extensions: $\F=\F_0\subset \F_1\subset\cdots\subset\F_{r+1}=\F_L$. The type $\ty_F$ stores monic irreducible polynomials $\psi_i(y)\in\F_i[y]$ such that $\F_{i+1}\simeq \F_i[y]/(\psi_i(y))$. We have $\psi_i(y)\ne y$, for all $i>0$. Finally, there are also \emph{residual polynomial} operators:
$$
R_i\colon k[x] \longrightarrow \F_i[y],\quad 0\le i\le r+1.
$$ 
For all $0\le i\le r$, we have $R_i(F)\sim \psi_i^{\omega_{i+1}}$ and $R_i(\phi_{i+1})\sim\psi_i$, where the symbol $\sim$ indicates that the polynomials coincide up to a multiplicative constant in $\F_i^*$. For $i=0$ we have $R_0(F)=\overline{F}=\psi_0^{\omega_1}$ and $R_0(\phi_1)=\overline{\phi_1}=\psi_0$. The exponents $\omega_{i+1}$ are all positive and $\omega_{r+1}=1$. The operator $R_{r+1}$ is defined only when $\phi_{r+1}\ne F$; in this case, we also have $R_{r+1}(F)\sim \psi_{r+1}$, with $\psi_{r+1}(y)\in\F_{r+1}[y]$ monic of degree one such that $\psi_{r+1}(y)\ne y$. 

We recall that the \emph{length} of a Newton polygon $N$ is the abscissa of its right end point; we denote it by $\ell(N)$. There is a strong link between the operators $R_i$ and $N_{i+1}^-$ \cite[Lem. 2.17,(2)]{HN}. 

\begin{lemma}\label{length}Let $0\le i\le r$.
For any non-zero polynomial $g(x)\in k[x]$, $\ell(N_{i+1}^-(g))$ is equal to the maximal exponent $\ord_{\psi_i}R_i(g)$ with which $\psi_i$ divides $R_i(g)$ in $\F_i[y]$. 
\end{lemma}


From all these data of the OM representation some more numerical invariants are deduced. Initially we take:
$$m_0:=1,\quad f_0:=\deg \psi_0,\quad e_0:=1,\quad h_0:=V_0:=\mu_0:=\nu_0=0.
$$Then we define for all $1\le i\le r+1$:
$$
\as{1.2}
\begin{array}{l}
h_i,\,e_i \ \mbox{ positive coprime integers such that }\la_i=-h_i/e_i,\\
f_i:=\deg \psi_i,\\
m_i:=\deg \phi_i=e_{i-1}f_{i-1}m_{i-1}=(e_0\,e_1\cdots e_{i-1})(f_0f_1\cdots f_{i-1}),\\
\mu_i:=\sum_{1\le j\le i}(e_jf_j\cdots e_if_i-1)h_j/(e_1\cdots e_j),\\
\nu_i:=\sum_{1\le j\le i}h_j/(e_1\cdots e_j),\\
V_i:=v_i(\phi_i)=e_{i-1}f_{i-1}(e_{i-1}V_{i-1}+h_{i-1})=(e_0\cdots e_{i-1})(\mu_{i-1}+\nu_{i-1}).
\end{array}
$$

The general definition of a type may be found in \cite[Sec. 2.1]{HN}. In later sections, we shall consider types which are not necessarily optimal nor $F$-complete. So, it may be convenient to distinguish these two properties among all features of a type that we have just mentioned.

\begin{definition}\label{optimal}Let  $\ty=(\psi_0;(\phi_1,\lambda_1,\psi_1);\cdots;(\phi_i,\lambda_i,\psi_i))$ be a type of order $i$. \medskip

\noindent{$\bullet$} We say that $\ty$ is \emph{optimal} if $m_1<\cdots<m_i$.
\medskip

\noindent{$\bullet$} We say that $\ty$ \emph{divides} the polynomial $g(x)\in K[x]$ (and we write $\ty\mid g$) if $\psi_i\mid R_i(g)$ in $\F_i[y]$. We say that $\ty$ is \emph{$g$-complete} if $\ord_{\psi_i}R_i(g)=1$.
\end{definition}

Thus, for a general type of order $i$ we have $m_1\mid\cdots\mid m_i$ and $\omega_i>0$, but not necessarily
$m_1<\cdots<m_i=\deg F$, and $\omega_i=1$. These were particular properties of our optimal and $F$-complete type $\ty_F$ of order $i=r+1$,
constructed from an Okutsu frame and an Okutsu approximation to $F$. 

An irreducible polynomial $F$ admits infinitely many different OM representations. However, the numeri\-cal invariants $e_i,f_i,h_i$, for $0\le i\le r$, and the MacLane valuations 
$v_1,\dots,v_{r+1}$ attached to $\ty_F$, are canonical invariants of $F$. 

The data $\lambda_{r+1},\psi_{r+1}$ are not invariants of $F$; they depend on the choice of the Okutsu approximation $\phi_{r+1}$. The integer slope $\lambda_{r+1}=-h_{r+1}$ measures how close is $\phi_{r+1}$ to $F$. We have $\phi_{r+1}=F$ if and only if $h_{r+1}=\infty$.

\begin{definition}
An \emph{Okutsu invariant} of $F(x)$ is a rational number that depends only on $e_1,\dots,e_r,f_1,\dots,f_r,h_1,\dots,h_r$.
\end{definition}

Several arithmetic invariants of the polynomial $F(x)$ and the field extension $L/k$ are Okutsu invariants of $F(x)$. For instance, let us define:
$$
\begin{array}{l}
e(F):=e(L/k),  \mbox{ the ramification index of }L/k,\\
f(F):=f(L/k), \mbox{ the residual degree of }L/k,\\
\mu(F):=\max\{v(g(\t))\mid g(x)\in\oo[x] \ \mbox{monic of degree less than }n\},\\
\ind(F):=\operatorname{length}_\oo(\oo_L/\oo[\t]),\quad \mbox{(length as an $\oo$-module)},\\
\mathfrak{f}(F):=\min\{\delta\in\Z_{\ge0}\mid \m_L^\delta \oo_L\subset \oo[\t]\},\\
\exp(F):=\min\{\delta\in\Z_{\ge0}\mid \m^\delta \oo_L\subset \oo[\t]\}.
\end{array}
$$

We abuse of language and say that $\ind(F)$ is the \emph{index} of $F(x)$, $\mathfrak{f}(F)$ is the \emph{conductor} of $F(x)$, and $\exp(F)$ is the \emph{exponent} of $F(x)$. 
Actually, $\mathfrak{f}(F)$ is the $\m_L$-adic exponent of the conductor of $\oo[\t]$ as an order of $\oo_L$, and $\exp(F)=\lceil \mathfrak{f}(F)/e(F)\rceil$. Also, if $k$ is a finite extension of the field $\Q_p$ of $p$-adic numbers, then $\ind(F)$ is related to the $v$-adic exponent of the finite index $(\oo_L\colon\oo[\t])$ of these abelian groups; more precisely, $\ind(F)=v((\oo_L\colon\oo[\t]))/[k\colon\Q_p]$.  

The next proposition gathers the results \cite[Cor. 3.8]{HN}, \cite[Thm. 5.2]{newapp} and \cite[Prop. 3.5]{GNP}. 

\begin{proposition}\label{okutsu}
$$\as{1.2}
\begin{array}{l}
e(F)=e_0\,e_1\cdots e_r,\\ f(F)=f_0f_1\cdots f_r,\\
\mu(F)=\mu_r=\sum_{1\le j\le r}(e_jf_j\cdots e_rf_r-1)h_j/(e_1\cdots e_j),\\
\exp(F)=\lfloor\mu_r\rfloor\\
\ind(F)=\dfrac{n}2\left(\mu_r-1+e(F)^{-1}\right).
\end{array}
$$
\end{proposition}

We end this section with a result extracted from \cite[Prop. 3.5,(5)]{HN}.

\begin{proposition}\label{vgt}
Let $g(x)\in\oo[x]$ a non-zero polynomial. For any $1\le i\le r$ take a line of slope $ \lambda_i$ far below $N_i(g)$, and shift it upwards till it touches the polygon for the first time. Let $(0,H)$ be the intersection point of this line with the vertical axis. Then,  
$$v(g(\t))\ge v_{i+1}(g)/(e_1\cdots e_i)=H/(e_0\,e_1\cdots e_{i-1}),$$ and equality holds if and only if $\psi_i\nmid R_i(g)$ in $\F_i[y]$.
\end{proposition}

We emphasize two particular instances of this result, that occur quite often.  

\begin{corollary}\label{previous}With the above notation,
\begin{enumerate}
\item If $\deg g<m_{i+1}$, then $v(g(\t))=v_{i+1}(g)/(e_1\cdots e_i)$, for all $1\le i\le r$.
\item $v(\phi_i(\t))=(V_i+|\lambda_i|)/(e_0\,e_1\cdots e_{i-1})=\mu_{i-1}+\nu_i$, for all $1\le i\le r+1$.
\end{enumerate}
\end{corollary}

\subsection{Different and discriminant}\label{subsecDiff}
Let $\df(L/k)$ be the different ideal of $L/k$ and $\dsc(L/k)=N_{L/k}(\df(L/k))$, the discriminant ideal of $L/k$. From now on we denote $$e:=e(F)=e(L/k),\qquad f:=f(F)=f(L/k).$$ The following identities are well-known:
\begin{equation}\label{known}
\as{1.3}
\begin{array}{l}
\dsc(F)=\res(F,F')=N_{L/k}(F'(\t)),\\
v(\dsc(F))=2\ind(F)+v_\m(\dsc(L/k)),\\
\df(L/k)=(\m_L)^{e-1+\rho},\quad \rho\ge0,
\end{array}
\end{equation}
and $\rho=0$ if and only if $L/k$ is tamely ramified.

Since the Okutsu invariants of $F$ can be computed from data contained in the OM representation of $F$, the only extra information we need to compute the different ideal is the value of the non-negative integer $\rho$, which is a measure of the wildness of the ramification. If $L/k$ is a Galois extension, then it is well-known how to deduce the value of $\rho$ from the higher ramification subgroups of the Galois group \cite{serre}. We are interested in the computation of $\rho$ for an arbitrary extension $L/k$. 

By applying the formula for $\ind(F)$ in Proposition \ref{okutsu}, we get 
a first formula for $\rho$ in terms of the Okutsu invariants $\mu(F)$, $f(F)$ and the $v$-value of $\dsc(F)$.

\begin{theorem}\label{vdisc}
$v(\dsc(F))=f(e\mu(F)+\rho)=n\mu(F)+f\rho$.
\end{theorem}

However, we are interested in a computation of $\rho$ that does not require the computation of $\dsc(F)$. This can be achieved as follows.

\begin{corollary}\label{rho}
$v_{r+1}(F')=v_{\m_L}(F'(\t))=e\mu(F)+\rho$.
\end{corollary}

\begin{proof}
Since $\deg F'<n=m_{r+1}$, item 1 of Corollary \ref{previous} and Proposition \ref{okutsu} show that $v(F'(\t))=v_{r+1}(F')/e$. Hence, by (\ref{known}) and Theorem \ref{vdisc},
$$
nv(F'(\t))=v(\dsc(F))=n\mu(F)+f\rho.
$$
If we divide both terms of the equality by $f$, we get the result.
\end{proof}

As a consequence, the conductor of $F(x)$ is an Okutsu invariant too. In fact, it is well-known that \cite[Ch.III,\S6,Cor.2]{serre}
$$
F'(\t)\oo_L=(\m_L)^{\mathfrak{f}(F)}\df(L/k).
$$
On the other hand, Corollary \ref{rho} shows that
$$
v_{\m_L}(F'(\t))=e\mu(F)+\rho=e\mu(F)+v_{\m_L}(\df(L/k))-e+1.
$$
Thus, by comparing the two identities we get an expression of $\mathfrak{f}(F)$ as an Okutsu invariant.

\begin{corollary}
 $\mathfrak{f}(F)=e\mu(F)-e+1=2\ind(F)/f$.
\end{corollary}

In practice, we deal with some polynomial $f(x)$ with coefficients in a Dedekind domain $A$, and $F(x)$ is an irreducible factor of $f(x)$ over the completion of $A$ at a nonzero prime ideal $\p$ of $A$. Thus, we never get an exact computation of the polynomial $F(x)$, but only an approximation to it. Therefore, if $L/k$ is wildly ramified and we want to apply Corollary \ref{rho} to compute $\rho$, we need to find a polynomial $\phi(x)\in A[x]$, sufficiently close to $F(x)$ to have $v_{r+1}(\phi')=v_{r+1}(F')$ and $\mu(\phi)=\mu(F)$. To this end is devoted the next section. 

\subsection{Okutsu approximations and the different}\label{subsecApprox}
Suppose $L/k$ is wildly ramified. Let $\phi(x)\in\oo[x]$ be an Okutsu approximation to $F(x)$. That is, $\phi$ is a monic separable polynomial of degree $n$ such that $v(\phi(\t))/n>v(\phi_r(\t))/m_r$, where $r$ is the Okutsu depth of $F(x)$ and $[\phi_1,\dots,\phi_r]$ is some Okutsu frame of $F(x)$. This polynomial $\phi(x)$ is always irreducible and it has the same Okutsu invariants as $F(x)$ \cite[Lem. 4.3]{okutsu}. In particular, Proposition \ref{okutsu} shows that \begin{equation}\label{equal}
e(\phi)=e(F),\quad f(\phi)=f(F),\quad\mu(\phi)=\mu(F).
\end{equation}

However, the non-negative integer $\rho(F):=v_{r+1}(F')-e(F)\mu(F)$ is not an Okutsu invariant of $F(x)$, and we may have $v_{r+1}(\phi')\ne v_{r+1}(F')$. The aim of this section is to find how close to $F$ must be taken the approximation $\phi$, to ensure that $v_{r+1}(\phi')=v_{r+1}(F')$. By Corollary \ref{rho} and (\ref{equal}), this implies $\rho(\phi)=\rho(F)$; thus, the polynomial $\phi$ may be used to determine the value of $\rho(F)$ and hence the different ideal of $L/k$.

For any Okutsu approximation $\phi$, we may apply the results of section \ref{subsecOkutsu} to $\phi_{r+1}:=\phi$.
Since $\deg\phi=\deg F$ and both polynomials are monic, the difference $a(x):=F(x)-\phi(x)$ has degree less than $n=m_{r+1}$. Hence, $F(x)=a(x)+\phi(x)$ is the canonical $\phi$-development of $F$, and the right end point of $N_{r+1}(F)$ is $(1,V_{r+1})$. 

Since $\ord_{\psi_r}R_r(F)=\omega_{r+1}=1$, Lemma \ref{length} shows that $\ell(N_{r+1}^-(F))=1$. Hence, $N_{r+1}^-(F)=N_{r+1}(F)$ is one-sided of slope $\lambda_{r+1}=-h_\phi$, for some positive integer $h_\phi:=v_{r+1}(a)-V_{r+1}$ (see Figure \ref{fignewton}). In particular, $h_{r+1}=h_\phi$, $e_{r+1}=1$, $f_{r+1}=1$.

\begin{figure}\caption{Newton polygon $N_{r+1}(F)$}\label{fignewton}
\setlength{\unitlength}{5mm}
\begin{picture}(13,9)
\put(4.85,7.3){$\bullet$}\put(6.85,3.3){$\bullet$}
\put(4,1){\line(1,0){6}}\put(5,0){\line(0,1){8.5}}
\put(5,7.5){\line(1,-2){2}}\put(5.02,7.5){\line(1,-2){2}}
\multiput(7,0.9)(0,.25){11}{\vrule height2pt}
\multiput(4.9,3.5)(.25,0){9}{\hbox to 2pt{\hrulefill }}
\put(6.1,5.5){\begin{footnotesize}$-h_\phi$\end{footnotesize}}
\put(2.6,7.3){\begin{footnotesize}$v_{r+1}(a)$\end{footnotesize}}
\put(3.5,3.3){\begin{footnotesize}$V_{r+1}$\end{footnotesize}}
\put(6.9,.2){\begin{footnotesize}$1$\end{footnotesize}}
\put(4.6,.2){\begin{footnotesize}$0$\end{footnotesize}}
\end{picture}
\end{figure}
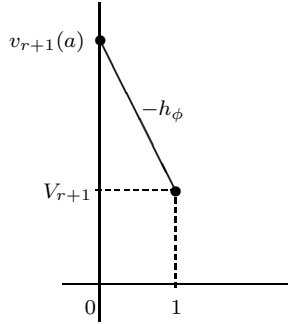

Clearly, $v(\phi(\t))$ is a measure of the quality of $\phi$ as an approximation to $F$. By item 2 of Corollary \ref{previous}, we have
$$
v(\phi(\t))=(V_{r+1}+h_\phi)/e=\mu_r+\nu_{r+1}.
$$
Since $V_{r+1}$ and $e$ are constant, the positive integer $h_\phi$ is a measure of the quality of the approximation too. The main result in this section
(Theorem \ref{howmuch}) finds a sufficient condition to ensure that $\phi$ is sufficiently close to $F$ to have the same exponent of the different ideal. 
Before proving this theorem we need a lemma, inspired by \cite{ore}.

\begin{lemma}\label{ore}
Let $F(x)\in\oo[x]$ be a monic irreducible separable polynomial, $\ty_F$ an OM representation of $F$ as in (\ref{OM}), and $\t\in \ks$ a root of $F$. 

Let $g(x)\in\oo[x]$ be a polynomial satisfying $\deg g<m_{i+1}$, for some $0\le i\le r$. Then,
$v(g'(\t))\ge v(g(\t))-\nu_i$. If $p\mid e_0\,e_1\cdots e_{i-1}$, then this inequality is strict.
\end{lemma}

\begin{proof}Let us prove the lemma by induction on $i$.
Suppose $i=0$. Recall that $\nu_0=0$ and $\overline{F}=(\psi_0)^{\omega_1}$, where $\psi_0$ is irreducible of degree $f_0=m_1$. Let $\pi\in\m$ be a uniformizer, and express $g(x)=\pi^aG(x)$, with $\overline{G}\ne0$.  Since $\deg\overline{G}\le \deg g<m_1=\deg\psi_0$, we have $\psi_0\nmid\overline{G}$, and this implies $v(G(\t))=0$. Since $g'(x)=\pi^aG'(x)$, we have then $v(g'(\t))\ge a=v(g(\t))$. For $i=0$, the empty product $e_0\,e_1\cdots e_{i-1}$ is equal to one; hence, the condition $p\mid e_0\,e_1\cdots e_{i-1}$ is never satisfied.

Let $i\ge1$, and suppose the statement of the lemma is true for all polynomials of degree less than $m_i$. Assume $m_i\le \deg g<m_{i+1}$, and let $$g(x)=\sum\nolimits_{0\le s}a_s(x)\phi_i(x)^s,\quad \deg a_s<m_i$$ be the canonical $\phi_i$-development of $g(x)$. 

For any $s\ge0$, let $(0,H_s)$ be the intersection point of the vertical axis with the line of slope $\lambda_i$ passing through $(s,v_i(a_s\phi_i^s))$. Clearly, $\deg a_s(\phi_i)^s\le \deg g<m_{i+1}$; hence, Proposition \ref{vgt} and Corollary \ref{previous} show that
$$v(a_s(\t)\phi_i(\t)^s)=H_s/(e_0\,e_1\cdots e_{i-1}),
$$
and $v(g(\t))=v_{i+1}(g)/e=H/(e_0\,e_1\cdots e_{i-1})$, where $(0,H)$ is the intersection point of the vertical axis with the line of slope $\lambda_i$ that first touches $N_i(g)$ from below (see Figure \ref{figH}). Therefore,
\begin{equation}\label{min}
v(g(\t))=\dfrac{H}{e_0\,e_1\cdots e_{i-1}}=\min_{0\le s}\left\{\dfrac{H_s}{e_0\,e_1\cdots e_{i-1}}\right\}=\min_{0\le s}\left\{v(a_s(\t)\phi_i(\t)^s)\right\}.
\end{equation}

\begin{figure}\caption{Computation of $v_{i+1}$}\label{figH}
\setlength{\unitlength}{5mm}
\begin{picture}(10,6.5)
\put(5.85,2.55){$\bullet$}\put(4.85,3.55){$\bullet$}
\put(0,0.7){\line(1,0){10}}
\put(6,2.7){\line(-1,1){1}}\put(6.02,2.7){\line(-1,1){1}}
\put(6,2.7){\line(3,-1){1}}\put(6.02,2.7){\line(3,-1){1}}
\put(5,3.7){\line(-1,2){1}}\put(5.02,3.7){\line(-1,2){1}}
\put(9,1.2){\line(-2,1){8}}
\put(.9,5.2){\line(1,0){.2}}
\put(1,-.3){\line(0,1){6.5}}
\put(5.9,.1){\begin{footnotesize}$s$\end{footnotesize}}
\put(-1.6,2.6){\begin{footnotesize}$v_i(a_s\phi_i^s)$\end{footnotesize}}
\put(5.6,3.7){\begin{footnotesize}$N_i^-(g)$\end{footnotesize}}
\put(8.5,1.6){\begin{footnotesize}$\lambda_i$\end{footnotesize}}
\put(-3.4,5.05){\begin{footnotesize}$H=v_{i+1}(g)/e_i$\end{footnotesize}}
\multiput(6.05,.6)(0,.25){9}{\vrule height2pt}
\multiput(.8,2.75)(.25,0){21}{\hbox to 2pt{\hrulefill }}
\end{picture}
\end{figure}
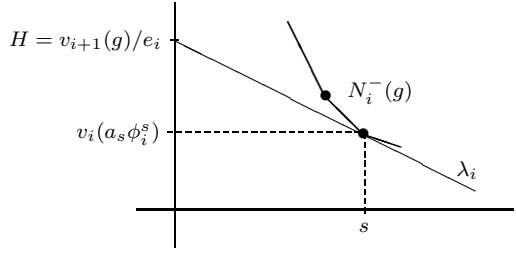

Let us write:
$$
\phi_i(x)\,g'(x)=\sum\nolimits_{0\le s\le m}b_s(x)\,\phi_i(x)^s,\quad b_s(x)=a_s'(x)\phi_i(x)+sa_s(x)\phi_i'(x). 
$$
Note that this is not the canonical $\phi_i$-development of $\phi_i(x)\,g'(x)$.
We claim that 
\begin{equation}\label{ba}
v(b_s(\t))\ge v(a_s(\t))+\mu_{i-1}, \quad \mbox{ for all }s\ge0,
\end{equation}
and if $p\mid e_0\,e_1\cdots e_{i-1}$, then the inequality is strict.

In fact, let us check that both summands of $b_s(\t)=a'_s(\t)\phi_i(\t)+sa_s(\t)\phi'_i(\t)$ satisfy this inequality. By the induction hypothesis and Corollary \ref{previous},(2),   
$$
v(a'_s(\t)\phi_i(\t))\ge v(a_s(\t))-\nu_{i-1}+\mu_{i-1}+\nu_i> v(a_s(\t))+\mu_{i-1}.
$$
On the other hand, $[\phi_1,\dots,\phi_{i-1}]$ is an Okutsu frame of $\phi_i(x)$
\cite[Cor. 2.5]{okutsu}, and Corollary \ref{rho} shows that $v_i(\phi_i')=e_0\,e_1\cdots e_{i-1}\mu_{i-1}+\rho(\phi_i)$. Since $\deg\phi'_i<m_i$, Corollary \ref{previous} shows that $v(\phi'_i(\t))=\mu_{i-1}+\rho(\phi_i)/(e_0\,e_1\cdots e_{i-1})$. Hence,
$$
v(sa_s(\t)\phi_i'(\t))\ge v(a_s(\t))+\mu_{i-1}+\rho(\phi_i)/(e_0\,e_1\cdots e_{i-1})\ge v(a_s(\t))+\mu_{i-1}, 
$$
and the inequality is strict if $p\mid e_0\,e_1\cdots e_{i-1}=e(\phi_i)$. This ends the proof of (\ref{ba}).

Now, from (\ref{ba}) and (\ref{min}), we get
\begin{align*}
v(\phi_i(\t)g'(\t))\ge&\ \min_{0\le s}\left\{v(b_s(\t)\phi_i(\t)^s)\right\}\\\ge&\
\min_{0\le s}\left\{v(a_s(\t)\phi_i(\t)^s)\right\}+\mu_{i-1}=v(g(\t))+\mu_{i-1}, 
\end{align*}
and the inequality is strict if $p\mid e_0\,e_1\cdots e_{i-1}$. This ends the proof of the lemma because $v(\phi_i(\t))=\mu_{i-1}+\nu_i$, by Corollary \ref{previous}.
\end{proof}

\begin{theorem}\label{howmuch}
Let $\phi$ be an Okutsu approximation to $F$, with $h_\phi>\rho(F)$. Then, $\rho(\phi)=\rho(F)$.
In particular, if $L'/k$ is the extension determined by $\phi$, then $\df(L/k)$ and $\df(L'/k)$ are divided by the same power of $\m_L$ and $\m_{L'}$,  respectively. 
\end{theorem}

\begin{proof}
Let $F(x)=a(x)+\phi(x)$, with $a(x)\in\oo[x]$ of degree less than $n=m_{r+1}$. Clearly, $v(a(\t))=v(\phi(\t))=\mu_r+\nu_{r+1}$. Hence, Lemma \ref{ore} shows that 
$$
v(a'(\t))\ge v(a(\t))-\nu_r=\mu_r+\nu_{r+1}-\nu_r=\mu_r+\dfrac{h_\phi}{e}> \mu_r+\dfrac{\rho(F)}e=v(F'(\t)),
$$
the last equality by Corollary \ref{rho}. From  $F'(x)=a'(x)+\phi'(x)$, we deduce $v(\phi'(\t))=v(F'(\t))$, or equivalently (by Corollary \ref{previous}), $v_{r+1}(\phi')=v_{r+1}(F')$.

Since $\phi$ and $F$ have the same Okutsu invariants, we have $e(\phi)=e(F)$ and $\mu(\phi)=\mu(F)$. By Corollary \ref{rho} we deduce that $\rho(\phi)=\rho(F)$.
\end{proof}

In order to get an algorithm to compute $\rho(F)$ (and $\df(L/k)$), we must replace the condition  $h_\phi>\rho(F)$ by something checkable. This is done in a different way for local fields of equal or unequal characteristic. 

\begin{corollary}\label{char0}
If $k$ has characteristic zero and $h_\phi\ge ev(e)$, then $\rho(\phi)=\rho(F)$.
\end{corollary}

\begin{proof}
It is well-known that in this case $\rho(F)\le ev(e)$. Hence, $h_\phi\ge ev(e)$ implies 
$h_\phi\ge \rho(F)$, and the arguments of the proof of Theorem \ref{howmuch} show that $v_{r+1}(\phi')\ge v_{r+1}(F')$.

Let us now exchange the role of $\phi$ and $F$. The list $[\phi_1,\dots,\phi_r]$ is an Okutsu frame of $\phi(x)$, and $F(x)$ is an Okutsu approximation to $\phi(x)$ \cite[Lem. 4.3]{okutsu}. Hence, $v_{r+1}(F)=V_{r+1}$. Also, $\phi=-a+F$ is the canonical $F$-development of $\phi$, so that $h_F=v_{r+1}(a)-V_{r+1}=h_\phi\ge ev(e)$. By the previous argument we have $v_{r+1}(F')\ge v_{r+1}(\phi')$, and hence,  $v_{r+1}(F')=v_{r+1}(\phi')$.
\end{proof}

In the equal characteristic case, the value of $\rho(F)$ is unbounded among all polynomials $F$ of a fixed degree. We may proceed as follows in this case.
 
\begin{corollary}\label{charp}
If $k$ has characteristic $p$ and $h_\phi> \rho(\phi)$, then $\rho(\phi)=\rho(F)$.
\end{corollary}

\begin{proof}
The argument of the proof of Theorem \ref{howmuch} shows that
$$
v(a'(\t))\ge \mu_r+\dfrac{h_\phi}{e}> \mu_r+\dfrac{\rho(\phi)}e=v(\phi'(\t)).
$$
Hence, $v(\phi'(\t))=v(F'(\t))$, and this implies $\rho(\phi)=\rho(F)$.
\end{proof}

Therefore, we may apply the Single-factor lifting (SFL) algorithm of \cite{GNP} to improve the Okutsu approximation $\phi$ till the checkable condition of either Corollary \ref{char0} or Corollary \ref{charp} is satisfied, and then we obtain a computation of $\rho(F)$ as  $\rho(\phi)$. In each iteration, from a given Okutsu approximation $\phi$, we compute a better one $\Phi$ satisfying $h_\Phi\ge 2h_\phi$; thus, only a finite (and small) number of iterations are required.  

In the equal characteristic case, the condition $h_\phi> \rho(\phi)$ must be attained after a finite number of iterations, because $\rho(F)$ is fixed and when we reach $h_\phi>\rho(F)$, then we have already $\rho(\phi)=\rho(F)<h_\phi$, by Theorem \ref{howmuch}.

\section{Computation of the $\p$-adic valuation of discriminants}
Let $A$ be a Dedekind domain whose field of fractions $K$ is a global field. Let $\p$ be a non-zero prime ideal of $A$, and $v_\p\colon K\lra \Z\cup\{\infty\}$ the canonical $\p$-adic valuation. 
Let $L=K[x]/(f(x))$ be the finite extension of $K$ determined by a monic irreducible separable polynomial $f(x)\in A[x]$, and let $ B$ be the integral closure of $A$ in  $L$. 

In this section, we apply the results of the preceding section to compute the $\p$-adic valuation of the discriminant ideal $\dsc(L/K)$, and the $v_\p$-value of the discriminant $\dsc(g)$ of an arbitrary polynomial $g(x)\in K[x]$. Both tasks are based on a combination of the Montes algorithm \cite{algorithm} and the Single-factor lifting (SFL) algorithm \cite{GNP}. 

\subsection{Local computation of the different}
Let $K_\p$ be the completion of $K$ at the prime ideal $\p$ and let $\oo_\p$ be the ring of integers of $K_\p$. We denote still by $v_\p$ the canonical extension of the $\p$-adic valuation to an algebraic closure of $K_\p$. Let 
$$
f(x)=F_1(x)\cdots F_t(x), 
$$
be the factorization of $f(x)$ into a product of monic irreducible polynomials in $\oo_\p[x]$. Let $L_1,\dots,L_t$ be the finite extensions of $K_\p$ obtained by adjoining to $K_\p$ a root of $F_1,\dots,F_t$, respectively. After a classical theorem of Hensel, there are exactly $t$ prime ideals $\P_1,\dots,\P_t$ in $ B$ lying over $\p$, and 
$$
e(\P_i/\p)=e(L_i/K_\p),\quad 
f(\P_i/\p)=f(L_i/K_\p),\quad 1\le i\le t. 
$$

Let us denote by $F_\P(x)\in\oo_\p[x]$ the irreducible factor that corresponds to any such prime ideal $\P$ dividing $\p$, and let $L_\P$ be the corresponding extension of $K_\p$. At the input of $f(x)$ and $\p$, the Montes algorithm computes an OM representation 
of every $\P\mid\p$, or equivalently, of every irreducible factor $F_\P$ \cite{newapp}:
\begin{equation}\label{OM2}
\ty_\P=(\psi_0;(\phi_1,\lambda_1,\psi_1);\cdots;(\phi_r,\lambda_r,\psi_r);(\phi_{r+1},\lambda_{r+1},\psi_{r+1})),
\end{equation}
carrying all invariants and operators described in section \ref{subsecOkutsu}. The polynomials $ \phi_1,\dots,\phi_{r+1}$ have all coefficients in $A$. Each type $\ty_\P$ is $f$-complete and $F_\P$-complete; it singles out $\P$ (or $F_\P$) by the following property:
$$
\ty_\P\mid F_\P,\quad \ty_\P\nmid F_\q, \ \forall\,\q\ne\P.
$$

In particular, all Okutsu invariants of $F_\P$, like $e(F_\P)$, $f(F_\P)$, $\mu(F_\P)$, $\ind(F_\P)$, $\mathfrak{f}(F_\P)$ and $\exp(F_\P)$, are computed by direct formulas in terms of data supported by the OM representation of $\P$.

\begin{remark}\label{precision}\rm
(1) \ Let $\phi(x)\in A[x]$ be an Okutsu approximation to $F_\P(x)$, and let $N_{r+1}$ be the Newton polygon operator with respect to $\phi(x)$ and the valuation $v_{r+1}$. By \cite[Thm.3.1]{HN}, we have $N_{r+1}^-(f)=N_{r+1}(F_\P)$. Thus, the local invariant  $h_\phi$, that was defined by $h_\phi:=v_{r+1}(F_\P-\phi)-V_{r+1}$, may be computed as well as
$h_\phi=v_{r+1}(a_0)-V_{r+1}$, where $a_0$ is the $0$-th coefficient of the canonical $\phi$-development of $f(x)$ (see Figure \ref{fignewton}).\medskip

(2) \ 
The restriction to $k$ of the valuation $v_{r+1}$ coincides with $e(\P/\p)v_\P$. Hence, the value of $v_{r+1}(a_0)$ does not change if we replace $\phi(x)$ by $\phi(x)+b(x)$, with $b(x)\in A[x]$ of degree less than $n_\P$ and  $b(x)\equiv 0 \md{\p^m}$, for some $m >(h_\phi+V_{r+1})/e(\P/\p)$. In particular, the coefficients of the Okutsu approximation $\phi$ may be always simplified modulo such a power of $\p$.
\end{remark}

At the input of an OM representation of $\P$ (or $F_\P$) and a prescribed precision $\nu\in \N$, the SFL algorithm computes an Okutsu approximation $\phi(x)$ to $F_\P(x)$, such that $\phi(x)\equiv F(x)\md{\p^\nu}$ \cite{GNP}. In each iteration of the main loop of SFL, a given Okutsu approximation $\phi(x)$ is used to construct a better Okutsu approximation $\Phi(x)$ with $h_\Phi\ge 2h_\phi$. By \cite[Lem. 4.5]{okutsu}, we reach the desired precision $\nu$ when $h_\phi\ge e(\P/\p)(\nu-\nu_r)$, where $r$ is the Okutsu depth of $F_\P$, and $\nu_r$ is the Okutsu invariant described in section \ref{subsecOkutsu}.

Therefore, the results of section \ref{secDiff} lead to the following routine to compute the different of the local extension $L_\P/K_\p$.\bigskip

\noindent{\bf Routine Different (char(K)\,=\,$\mathbf{0}$)}\vskip 1mm

\noindent INPUT:

\noindent $-$ A monic irreducible polynomial $f(x)\in A[x]$.

\noindent $-$ An OM representation $\ty_\P$, as in (\ref{OM2}), of a prime ideal $\P$ of $ B$.

\vskip 2mm

\st1  $e\leftarrow e_0\,e_1\cdots e_r$, \quad
$\mu\leftarrow \sum_{1\le j\le r} (e_jf_j\cdots e_rf_r-1)h_j/(e_1\cdots e_j)$.

\st2 {\tt precision} $\leftarrow e\,v_\p(e)$,\quad $\rho\leftarrow 0$.

\st3 IF {\tt precision} $\ne0$ THEN DO

\stst{3.1} Apply the SFL algorithm to compute an Okutsu approximation $\phi\in A[x]$ to $F_\P$, such that $h_\phi\ge$ {\tt precision}. 

\stst{3.2} Compute  $\rho\leftarrow v_{r+1}(\phi')-\,e\cdot\mu$. (use Proposition \ref{vgt})

\st4  \ {\tt Diff} $\leftarrow e-1+\rho$.
\vskip 2mm

\noindent OUTPUT:

\noindent $-$ The $\P$-valuation of $\df(L_\P/K_\p)$, as the value of the variable {\tt Diff}.\bigskip

\noindent{\bf Routine Different (char(K)\,=\,$\mathbf{p>0}$)}\vskip 1mm

\noindent INPUT:

\noindent $-$ A monic irreducible polynomial $f(x)\in A[x]$.

\noindent $-$ An OM representation $\ty_\P$, as in (\ref{OM2}), of a prime ideal $\P$ of $ B$.

\vskip 2mm

\st1  $e\leftarrow e_0\,e_1\cdots e_r$, \quad
$\mu\leftarrow \sum_{1\le j\le r} (e_jf_j\cdots e_rf_r-1)h_j/(e_1\cdots e_j)$.

\st2  $\rho\leftarrow 0$,\quad$\phi\leftarrow\phi_{r+1}$,\quad$h\leftarrow h_{r+1}$.

\st3 IF $p\mid e$ THEN DO

\stst{3.1} Compute $\rho\leftarrow v_{r+1}(\phi')-\,e\cdot\mu$. (use Proposition \ref{vgt})

\stst{3.2} WHILE $\,h\le\rho\,$ DO

\ststst{3.2.1} Apply one iteration of the SFL algorithm to compute an Okutsu approximation $\Phi\in A[x]$ to $F_\P(x)$, such that $h_\Phi\ge 2h_\phi$. 

\ststst{3.2.2} Compute   $\rho\leftarrow v_{r+1}(\Phi')-\,e\cdot\mu$. (use Proposition \ref{vgt})

\ststst{3.2.3} $h\leftarrow h_\Phi$.

\st4  \ {\tt Diff} $\leftarrow e-1+\rho$.
\vskip 2mm

\noindent OUTPUT:

\noindent $-$ The $\P$-valuation of $\df(L_\P/K_\p)$, as the value of the variable {\tt Diff}.\bigskip

Let us analyze the complexity of this routine. All tasks we are interested in may be performed modulo $\p^\nu$, for a sufficiently high precision $\nu$. Thus, we may assume that the elements of $A$ are finite $\pi$-adic developments, for some $\pi\in A$ which is a local generator of $\p$. In particular, the computation of the $\p$-adic valuation $v_\p$ has a negligible cost. 

\begin{definition}
An operation of $A$ is called \emph{$\p$-small} if it involves two elements belonging to a fixed system of representatives of $A/\p$.
\end{definition}

Let $q:=\#A/\p$. A $\p$-small operation is equivalent to  $O\left(\log(q)^{1+\epsilon}\right)$ word ope\-rations, the cost of an operation in the residue field $A/\p$.
Working at precision $\nu$, each multiplication in $A$ costs $O(\nu^{1+\epsilon})$
$\p$-small operations if we assume the fast multiplications techniques of 
Sch\"onhage-Strassen \cite{SS}.

Denote from now on,
$$n:=[L\colon K], \quad e:=e(\P/\p),\quad f:=f(\P/\p),\quad n_\P:=ef, \quad \delta_\P:=v_\p(\dsc(F_\P)).
$$
Note that $n,e,f,n_\P$ are intrinsic invariants of $L/K$, but $\delta_\P$ depends on the choice of the polynomial $f(x)\in A[x]$ used to construct this extension.   
In the tamely ramified case the routine is trivial; thus, we assume from now on that the characteristic $p$ of $A/\p$ divides $e$. In particular, the different ideal $\df(L_\P/K_\p)$ is divided by the $e$-th power of the maximal ideal, and $n_\P\le v_\p(\dsc(L_\P/K_\p))\le\delta_\P$.

\begin{lemma}\label{complexitySFL}
The SFL algorithm requires $
O\left(nn_\P (h/e)^{1+\epsilon}+n(\delta_\P)^{1+\epsilon}\right)
$
$\p$-small operations in $A$ to compute an Okutsu approximation $\phi$ to $F_\P$, such that $h_\phi\ge h$. 
\end{lemma}

\begin{proof}
Let $r$ be the Okutsu depth of $F_\P$ and consider the Okutsu invariants $e_j$, $f_j$, $\mu_j$, $\nu_j$, $V_j$, described in section \ref{subsecOkutsu}. Along the proof of \cite[Lem. 6.5]{GNP}, it is obtained an estimation of $O\left(nn_\P((h/e)^{1+\epsilon}+(V_{r+1}/e)^{1+\epsilon})\right)$ $\p$-small operations for the SFL routine. The lemma is then a consequence of
\begin{equation}\label{Vdisc}
V_{r+1}/e=\mu_r+\nu_r\le 2\mu_r\le 2\delta_\P/n_\P.
\end{equation}
The inequality $\nu_r\le \mu_r$ is a consequence of $e_if_i>1$, for all $1\le i\le r$, and the explicit formulas for these Okutsu invariants given in section \ref{subsecOkutsu}. The inequality $\mu_r=\mu(F_\P)\le \delta_\P/n_\P$ follows from Proposition \ref{okutsu} and Theorem \ref{vdisc}.
\end{proof}

By \cite[Lem. 4.21]{HN}, the computation of $v_{r+1}(\phi')$ is essentially equivalent to the computation of the $(\phi_1,\dots,\phi_r)$-multiadic expansion of $\phi'$. By \cite[Lem. 5.4]{BNS}, this requires $O((n_\P)^{1+\epsilon})$ operations in $A$.
By Remark \ref{precision},(2) and (\ref{Vdisc}), this is equivalent to 
\begin{equation}\label{vr+1}
O\left((n_\P)^{1+\epsilon}((h_\phi+V_{r+1})/e)^{1+\epsilon}\right)= O\left((n_\P)^{1+\epsilon}(h_\phi/e)^{1+\epsilon}+(\delta_\P)^{1+\epsilon}\right)
\end{equation}
$\p$-small operations in $A$. 

\begin{theorem}\label{complexityDiff}
Assuming $\P$ wildly ramified over $\p$, the  
routine {\tt Different} has a cost of 
$O\left(n(\delta_\P)^{1+\epsilon}\right)$  $\p$-small operations in $A$
\end{theorem}

\begin{proof}
Suppose first $\op{char}(K)=0$. By Lemma \ref{complexitySFL}, step {\bf 3.1} has a complexity of
$$
O\left(n n_\P v_\p(e)^{1+\epsilon}+n(\delta_\P)^{1+\epsilon}\right)=
O\left(n (n_\P)^{1+\epsilon}+n(\delta_\P)^{1+\epsilon}\right)=
O\left(n(\delta_\P)^{1+\epsilon}\right),
$$
the last equality because $n_\P\le \delta_\P$. By (\ref{vr+1}), the complexity of step {\bf 3.2} is dominated by that of step {\bf 3.1}. This ends the proof of the lemma in this case.

If $\op{char}(K)=p>0$, the routine {\tt Different} is less efficient in practice because we must compute $v_{r+1}(\phi')$ after every iteration of the SFL routine. By the remarks following Corollary \ref{charp}, the number of iterations is bounded from above by $\log_2(\rho)$, where $\rho:=\rho(F_\P)$. 

By Lemma \ref{complexitySFL}, the complexity of all SFL iterations is
$$
O\left(nn_\P(\rho/e)^{1+\epsilon}+n(\delta_\P)^{1+\epsilon}\right)=
O\left(n (\delta_\P)^{1+\epsilon}\right).
$$
In the last equality we used $n_\P\rho/e=f\rho\le\delta_\P$, by Theorem \ref{vdisc}.

Assume that we start the iterations with $h_\phi=1$ (the worst possible case). If we compute each Okutsu approximation $\phi$ at the precision described in Remark \ref{precision}, the complexity of the computation of all $v_{r+1}(\phi')$ is, by (\ref{vr+1}):
$$
\sum_{h=1}^{\log_2 \rho}(n_\P2^h/e)^{1+\epsilon}+(\delta_\P)^{1+\epsilon}=
O\left((n_\P\rho/e)^{1+\epsilon}+(\delta_\P)^{1+\epsilon}\log_2 \rho\right)=O\left((\delta_\P)^{1+\epsilon}\right).
$$
\end{proof}

\subsection{Local computation of discriminants}\label{subsecDiscf}
The discriminant of $L/K$ is the product of the local discriminants, so that
$$v_\p(\dsc(L/K))=\sum\nolimits_{\P\mid \p}v_\p(\dsc(L_\P/K_\p)).$$
Since the local discriminant ideal $\dsc(L_\P/K_\p)$ is the norm of the local different ideal $\df(L_\P/K_\p)$, the routine {\tt Different} leads in an obvious way to a routine to compute $v_\p(\dsc(L/K))$. This routine does not require the previous computation of the discriminant $\dsc(f)$ of the polynomial $f(x)$; actually $v_\p(\dsc(f))$ may be deduced from the identity $v_\p(\dsc(f))=v_\p(\dsc(L/K))+2\ind_\p(f)$, where  
\begin{equation}\label{indp}
\ind_\p(f):=\operatorname{length}_{\oo_\p}( B\otimes_{A}\oo_\p)/(A[\t]\otimes_{A}\oo_\p).
\end{equation}
This local index $\ind_\p(f)$ is computed by the Montes algorithm as a by-product. 

This allows us to consider a more general routine {\tt pDiscriminant} that admits an arbitrary monic polynomial $g(x)\in A[x]$   
as input. Before describing the routine, let us review some generalities on discriminants of polynomials and the role of the local index $\ind_\p(g)$ for a non-irreducible polynomial.

Let $g(x)=a_nx^n+\cdots +a_1x+a_0$ be a polynomial of degree $n$, with coefficients in $K$. The discriminant of $g(x)$ is defined as:
$$
\dsc(g):=a_n^{2n-2}\prod\nolimits_{i\ne j}(\t_i-\t_j)=a_n^{-1}\res(g,g'),
$$ 
where $\t_1,\dots,\t_n$ are the $n$ roots of $g(x)$ in an algebraic closure, with due counting of multiplicities. Clearly, for arbitrary $a,b\in K^*$, one has
$$
\dsc(ag(bx))=a^{2n-2}b^{n^2-n}\dsc(g(x)).
$$
Thus, as regards the computation of $\dsc(g)$, we may assume from now on that $g(x)$ is a monic polynomial with coefficients in $A$.

Let $g(x)=G_1(x)\cdots G_t(x)$ be the factorization of $g(x)$ into a product of monic irreducible polynomials in $\oo_\p[x]$. Let $L_G$ be the finite extension of $K_\p$ determined by each irreducible factor $G(x)$, and denote by $\P_G$ the maximal ideal of the ring of integers of $L_G$. 

The discriminant has a well-known good behaviour with respect to products:
\begin{equation}\label{product}
\dsc(g)=\prod\nolimits_{1\le i\le t}\dsc(G_i)\cdot\prod\nolimits_{i\ne j}\res(G_i,G_j).
\end{equation}

We define the \emph{$\p$-index} of $g(x)$ to be:
\begin{equation}\label{pindex}
\ind_\p(g):=\sum\nolimits_{1\le i\le t}\ind_\p(G_i)+\sum\nolimits_{1\le i<j\le t}v_\p(\res(G_i,G_j)),
\end{equation}
where $\ind_\p(G_i)$ is, by definition, the local index $\ind(G_i)$ that was considered in section \ref{secDiff}. The reader may check that this definition coincides with that of (\ref{indp}) when $g(x)$ is irreducible in $A[x]$.  

For each $i$ we have $v_\p(\dsc(G_i))=2\ind_\p(G_i)+v_\p(\dsc(L_{G_i}/K_\p))$. Therefore, (\ref{product}) and (\ref{pindex}) show that:
$$
v_\p(\dsc(g))=2\ind_\p(g)+\sum\nolimits_{1\le i\le t}v_\p(\dsc(L_{G_i}/K_\p)).
$$
Now, if $g(x)$ is a separable polynomial, the Montes algorithm computes OM representations of all these local factors and also the value of $\ind_\p(g)$ as a by-product. Thus, properly combined with the routine {\tt Different}, we get the following routine to compute $\sum_{1\le i\le t}v_\p(\dsc(L_{G_i}/K_\p))$ and $v_\p(\dsc(g))$.\bigskip

\noindent{\bf Routine pDiscriminant}\vskip 1mm

\noindent INPUT:

\noindent $-$ A monic polynomial $g(x)\in A[x]$ such that  $\dsc(g)\ne0$.

\noindent $-$ A non-zero prime ideal $\p$ of $A$.

\vskip 2mm

\st1  Apply the Montes algorithm to get OM representations of the different irreducible factors $G(x)\in\oo_\p[x]$ of $g(x)$, and the value of $\ind_\p(g)$.

\st 2 {\tt Disc} $\leftarrow 0$.

\st3 FOR each factor $G$, with OM representation $\ty_G$, DO

\stst{3.1} $f\leftarrow f_0\,f_1\cdots f_r$.

\stst{3.2} Call {\tt Different}($g$,$\,\ty_G$) and accumulate $f\cdot v_{\P_G}(\df(L_G/K_\p))$ to {\tt Disc}.

\st4  {\tt DiscPol} $\leftarrow$ {\tt Disc} $+\,2\ind_\p(g)$.
\vskip 2mm

\noindent OUTPUT:

\noindent $-$ The value of {\tt Disc} is the sum of the $\p$-adic valuations of all local discriminants $\dsc(L_G/K_\p)$, for $G$ running on all irreducible factors of $g$ over $K_\p$. 
The value of {\tt DiscPol} is the $\p$-valuation of $\dsc(g)$. \medskip

\begin{theorem}\label{complexityDisc}
Let $n:=\deg g$, $\delta:=v_\p(\dsc(g))$ and $q:=\#A/\p$. The complexity of the routine {\tt pDiscriminant} is
$O\left(n^{2+\epsilon}+n^{1+\epsilon}(1+\delta)\log q+n^{1+\epsilon}\delta^{2+\epsilon}\right)$
$\p$-small operations in $A$.
\end{theorem}

\begin{proof}
The complexity of the Montes algorithm was analyzed in \cite{FV}, \cite{pauli}. A  sharper estimation has been obtained in \cite{BNS}, according to which, step {\bf 1} of the routine {\tt pDiscriminant} has a cost of  $O\left(n^{2+\epsilon}+n^{1+\epsilon}(1+\delta)\log q+n^{1+\epsilon}\delta^{2+\epsilon}\right)$ $\p$-small operations in $A$. 

By Theorem \ref{complexityDiff}, the complexity of step {\bf3} may be estimated as: 
$$O\left(\sum\nolimits_{\P}n(\delta_\P)^{1+\epsilon}\right)=O\left(n\delta^{1+\epsilon}\right) \quad\mbox{$\p$-small operations},
$$
where $\P$ runs on the wildly ramified primes lying over $\p$. Clearly, this cost is dominated by that of step {\bf1}.
\end{proof}

\begin{corollary}\label{complexityFast}
If we assume $\p$ small ($\log q=O(1)$), then the complexity of {\tt pDiscriminant} is $O\left(n^{2+\epsilon}+n^{1+\epsilon}\delta^{2+\epsilon}\right)$ word operations.
\end{corollary}

In this routine the input polynomial must have $\dsc(g)\ne0$, otherwise the Montes algorithm enters into an endless loop. Nevertheless, if $\chr(K)=0$, we can use upper bounds for the discriminant to design an algorithm that works for an arbitray monic input $g(x)\in A[x]$. For instance, if $A=\Z$ and $\p=p\Z$, for some prime number $p$, then Mahler's bound \cite{mahler}:
$$
|\dsc(g)|<n^n\|g\|_\infty^{2n-2},\quad \|g\|_\infty:=|a_0|+\cdots+|a_n|
$$
leads to $v_p(\dsc(g))< n\,v_p(n)+(2n-2)v_p(\|g\|_\infty)$. The reader may derive similar upper bounds, $v_\p(\dsc(g))< N(g)$, in the general case. Along the flow of Montes algorithm, partial $\p$-indices (taking positive integer values) are accumulated to a variable {\tt Index}, whose final output value is $\ind_\p(g)$. If $\dsc(g)=0$, we run into an endless loop that strictly increases {\tt Index} at each iteration. Thus, we may introduce a control instruction that allows the next iteration while the value of {\tt Index} is less than $N(g)$, but it breaks the loop and outputs $v_\p(\dsc(g))=\infty$ otherwise.

\section{Computation of the $\p$-adic valuation of resultants}
Consider two polynomials 
$$f(x)=a_nx^n+\cdots+a_1x+a_0,\quad g(x)=b_mx^m+\cdots+b_1x+b_0\in K[x],$$ of degree $n$, $m$, respectively. For any $a,b,c\in K^*$ we have
$$
\res(af(bx),cg(bx))=a^mc^nb^{nm}\res(f(x),g(x)).
$$
Hence, as regards the computation of the resultant, we may suppose that $f(x),g(x)$ are monic and have coeficients in $A$.   

In \cite[Sec. 4.1]{HN} the $\p$-adic valuation of the resultant of two monic polynomials in $A[x]$ is computed by applying a kind of non-optimized version of the Montes algorithm. In this section we reformulate those ideas into a concrete algorithm based on the optimized Montes algorithm as described in \cite{algorithm}. 

\subsection{Partial exponents of resultants}\label{subsecPartial}In this section we introduce certain positive integers $\res^h_{\ty,\phi_i}(f,g)$, whose accumulation is equal to $v_\p(\res(f,g))$. These partial $v_\p$-values of resultants are computed in terms of combinatorial data of adequate Newton polygons of the polynomials $f(x)$, $g(x)$. 

Let us describe in some detail the relevant Newton polygon routine. Along the algorithm that computes $v_\p(\res(f,g))$, we shall construct types of order $i-1$,
$$\ty=(\psi_0;(\phi_1,\lambda_1,\psi_1);\cdots;(\phi_{i-1},\lambda_{i-1},\psi_{i-1})),$$ dividing $f(x)$. All polynomials $\phi_1,\dots,\phi_{i-1}$ will have coefficients in $A$, and the type will always be loaded with three level $i$ invariants, $\phi_i$, $\cs_i$, $\omega_i$. The invariant $\phi_i$ is a \emph{representative} of $\ty$; that is, $\phi_i(x)\in A[x]$ is a monic polynomial of degree $m_i:=e_{i-1}f_{i-1}m_{i-1}$, such that $R_{i-1}(\phi_i)\sim\psi_{i-1}$. The choice of a
representative of $\ty$ determines a Newton polygon operator  $N_i$, with respect to  $\phi_i$ and the MacLane valuation $v_i$ supported by $\ty$ (see section \ref{subsecOkutsu}).

The invariant $\cs_i$ is a \emph{cutting slope}; it is a positive integer $h$ telling us that we must compute only the subpolygon $N_i^h(f)$ formed by the sides of $N_i(f)$ of slope less than $-h$. By a variation of Lemma \ref{length}, we are able to compute a priori the length of this subpolygon, and this is the role of the third invariant: $\omega_i=\ell(N_i^{\cs_i}(f))$. Since we always know a priori the length of the Newton polygons we are interested in, we may use the following {\tt Newton} routine.\medskip

\noindent{\bf Routine Newton($\ty$,\,$\omega$,\,$g$)}\vskip 1mm

\noindent INPUT:

\noindent $-$ A type $\ty$ of order $i-1\ge 0$ and a representative $\phi$ of $\ty$.

\noindent $-$ A positive integer $\omega$.

\noindent $-$ A non-zero polynomial $g(x)\in K[x]$.
\vskip 2mm

\noindent Compute the first $\omega+1$ coefficients $a_0(x),\dots,a_\omega(x)$ of the canonical $\phi$-expansion of $g(x)$ and the Newton polygon $N$ of the set of points $(s,v_i(a_s\phi^s))$, for $0\le s\le\omega$.
\vskip 2mm

\noindent OUTPUT:

\noindent $-$ A list $[S_1,\dots,S_t]$ of the sides of $N$.\medskip

A side $S$ of $N$ is a segment of negative slope of the Euclidean plane, whose end points have non-negative integer coordinates. The \emph{length}, $E=\ell(S)$, and \emph{height}, $H$, of $S$ are the lengths of its projections to the horizontal and vertical axes, respectively. Following the convention of \cite[Sec.1.1]{HN}, we admit sides $S$ of slope $-\infty$; we may think that the left end point of such an $S$ is $(0,\infty)$, and the right end point is $(s,u)$, with $s\in\Z_{>0}$ and $u\in\Z_{\ge0}$. The length of such a side is $E=s$ and the height is $H=\infty$ (see Figure \ref{figlength}).

If the left end point of a Newton polygon $N$ has a positive abscissa, we consider (formally) that the side of slope $-\infty$ determined by this point is also one of the sides of $N$. According to this formalism, if $N=S_1+\cdots+S_t$ are all sides, finite and infinite, of $N$, we have always $\ell(N)=\ell(S_1)+\cdots+\ell(S_t)$.

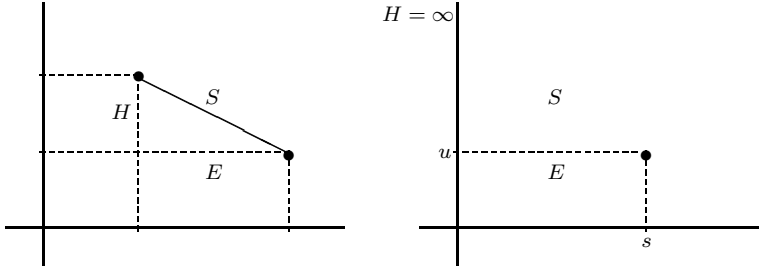
\begin{figure}\caption{Length and height of a side of negative slope}
\label{figlength}
\setlength{\unitlength}{5mm}
\begin{picture}(18,7.4)
\put(2.35,4.35){$\bullet$}\put(6.35,2.25){$\bullet$}
\put(-1,0.5){\line(1,0){9}}\put(0,-.5){\line(0,1){7}}
\put(2.5,4.5){\line(2,-1){4}}\put(2.52,4.5){\line(2,-1){4}}
\multiput(-.1,4.55)(.25,0){11}{\hbox to 2pt{\hrulefill }}
\multiput(-.1,2.5)(.25,0){27}{\hbox to 2pt{\hrulefill }}
\multiput(6.5,.4)(0,.25){9}{\vrule height2pt}
\multiput(2.5,.41)(0,.25){17}{\vrule height2pt}
\put(4.3,3.8){\begin{footnotesize}$S$\end{footnotesize}}
\put(4.3,1.8){\begin{footnotesize}$E$\end{footnotesize}}
\put(1.8,3.4){\begin{footnotesize}$H$\end{footnotesize}}
\put(15.85,2.25){$\bullet$}
\put(10,0.5){\line(1,0){9}}\put(11,-.5){\line(0,1){7}}
\multiput(10.9,2.5)(.25,0){20}{\hbox to 2pt{\hrulefill }}
\multiput(16,.4)(0,.25){9}{\vrule height2pt}
\put(13.4,3.8){\begin{footnotesize}$S$\end{footnotesize}}
\put(13.4,1.8){\begin{footnotesize}$E$\end{footnotesize}}
\put(15.9,0){\begin{footnotesize}$s$\end{footnotesize}}
\put(10.5,2.35){\begin{footnotesize}$u$\end{footnotesize}}
\put(9,6){\begin{footnotesize}$H=\infty$\end{footnotesize}}
\end{picture}
\end{figure}

\begin{definition}
Let $S,S'$ be two sides of negative slope, with lengths $E,E'$ and heights $H,H'$, respectively. Let $h$ be a non-negative integer.
We define
$$
\res(S,S'):=\min\{EH',E'H\},\qquad
\res^h(S,S'):=\res(S,S')-hEE'.
$$
Let $\ty=(\psi_0;(\phi_1,\lambda_1,\psi_1);\cdots;(\phi_{i-1},\lambda_{i-1},\psi_{i-1}))$ be a type of order $i-1$ and let $\phi_i$ be a representative of $\ty$. Let $f(x),g(x)\in A[x]$ be two monic polynomials. We define

\begin{align*}
\res^h_{\ty,\phi_i}(f,g):=&\ f_0\cdots f_{i-1}\sum\nolimits_{S\in N^h_i(f),\,T\in N^h_i(g)}\res^h(S,T)\\=&\ f_0\cdots f_{i-1}\left[\left[\sum_{S\in N^h_i(f), T\in N^h_i(g)}\res(S,T)\right]-h\ell(N^h_i(f))\ell(N^h_i(g))\right].
\end{align*}
\end{definition}

As the notation suggests, $\res^h_{\ty,\phi_i}(f,g)$ depends on the choice of the representative $\phi_i$.
The value of $\res^h_{\ty,\phi_i}(f,g)$ is equal to infinity if and only if both Newton polygons $N_i(f)$, $N_i(g)$
have a side of slope $-\infty$. This happens if and only if both polynomials $f,g$ are divisible by $\phi_i$ in $A[x]$.

If  $\res^h_{\ty,\phi_i}(f,g)$ is finite, then it takes a non-negative integer value. It vanishes if and only if
either $N_i(f)$ or $N_i(g)$ have no side (neither finite nor infinite) of slope less than $-h$; that is, if and only if either $N^h_i(f)$ or $N^h_i(g)$ have length zero.

\subsection{The algorithm}\label{subsecAlgorithm}
Roughly speaking, the algorithm consists of a simultaneous application of the Montes algorithm to $f(x)$ and $g(x)$, and the accumulation of all values $\res^{\cs_i}_{\ty,\phi_i}(f,g)$, for all the types $\ty$ considered along the flow of the algorithm, such that $\ty$ divides both polynomials $f(x)$ and $g(x)$.  

The Montes algorithm is fully described in \cite{algorithm}, in terms of the theoretical background developed in \cite{HN}. For a short review, we address the reader to \cite{survey}, or \cite[Sec. 4]{BNS}. 

We present a detailed description of the routine {\tt pResultant}, directly in pseudocode. We denote by $\phi_i^\ty$, $\lambda_i^\ty$, $\psi_i^\ty$, $V_i^\ty$, etc. the data at the $i$-th level of a type $\ty$. \bigskip 

\noindent{\bf Routine pResultant}\vskip 1mm

\noindent INPUT:

$-$ Two monic polynomials $f(x),g(x)\in A[x]$, with $\res(f,g)\ne0$,  $\deg f\le \deg g$.

$-$ A non-zero prime ideal $\p$ of $A$.

\vskip 2mm

\st1 Set  {\tt Resvalue} $\leftarrow0$.

\st2  Factorize $f(y)\md{\p}=\prod_{\varphi}\varphi(y)^{a_\varphi}$\, over $(A/\p)[y]$.

\st3
FOR every irreducible factor $\varphi$ DO

\stst4 Set $b\leftarrow \ord_\varphi(g \md{\p})$. IF $b=0$ THEN continue to the next factor $\varphi$.

\stst5 Take a monic $\phi(x)\in A[x]$ such that $\phi(y)\md{\p}=\varphi(y)$. 

\stst6 Create a type $\ty$ of order zero with: $\, \psi_0^\ty\leftarrow\varphi$, $\phi_1^\ty\leftarrow \phi$, $\omega_1^\ty\leftarrow a_\varphi$, $\alpha_1^\ty\leftarrow b$, and
\vskip -1mm\stst{}\quad $\cs_1^\ty\leftarrow0$.
Initialize the list \texttt{Stack} $\leftarrow[\ty]$.
\vskip 1mm

\stst{}\!\!WHILE \texttt{Stack} is non-empty DO

\ststst1 Extract (and delete) the last  type $\ty_0$ from \texttt{Stack}. Let $i-1$ be its order.

\ststst2 Call {\tt Newton}($\ty_0$, $\omega_i$, $f(x)$) and {\tt Newton}($\ty_0$, $\alpha_i$, $g(x)$).

\ststst3 Compute $\res_{\ty_0,\phi_i}^{\cs_i}(f,g)$ and add this number to {\tt Resvalue}.

\ststst4 FOR every finite side $S$ of $N_i^{\cs_i}(f)$ DO

\stststst5  Set $\lambda_i^{\ty_0}\leftarrow$ slope of $S$. IF $\lambda_i^{\ty_0}$ is not a slope of $N_i^{\cs_i}(g)$ THEN continue 
\vskip -1mm
\stststst{}\quad to the next side $S$.
 
\stststst6 Compute $R_i(f),R_i(g)\in \F_i[y]$. Factorize $R_i(f)$ over $\F_i[y]$. 

\stststst7
FOR every monic irreducible factor $\psi$ of $R_i(f)$ DO

\ststststst{8} Set $b\leftarrow\ord_\psi R_i(g)$.  IF $b=0$ THEN continue to the next factor $\psi$.
 
\ststststst{9} IF $\omega_{i}>1$ THEN make a copy $\ty$ of the type $\ty_0$ and extend it to an 
\vskip -1mm
\ststststst{}\qquad \ order $i$ type by setting: $\,\psi_i^{\ty}\leftarrow \psi$, $\ \omega_{i+1}^\ty\leftarrow \ord_{\psi}R_i(f)$, $\alpha_{i+1}^\ty\leftarrow b$. 
\vskip -1mm
\ststststst{}\qquad \ Compute a representative  $\phi\in A[x]$ of $\ty$.

\ststststst{}\quad ELSE apply one iteration of SFL to compute $\phi$ with $h_\phi\ge 2h_{\phi_i^\ty}$.

\ststststst{\!\!\!10}  IF $\deg \phi=\deg \phi_i^\ty$ THEN set 
\vskip-1mm\ststststst{}\qquad\quad 
$\phi_i^{\ty}\leftarrow \phi, \ \cs_i^{\ty}\leftarrow|\lambda_i^{\ty}|,\ \omega^{\ty}_i\leftarrow\omega^\ty_{i+1},\ \alpha^{\ty}_i\leftarrow\alpha^\ty_{i+1}$, 
\vskip-1mm\ststststst{}\qquad\quad  and delete all data in the $(i+1)$-th level of $\ty$ 
\vskip-1mm\ststststst{}\quad ELSE \,set $\phi_{i+1}^\ty\leftarrow\phi$, $\cs_{i+1}^\ty\leftarrow0$. 

\ststststst{\!\!\!11} Add $\ty$ to the \texttt{Stack}.

\stst{}\!\!END WHILE\vskip 2mm

\noindent OUTPUT:

\noindent $-$ The $\p$-adic valuation of $\res(f,g)$, as the value of the variable  {\tt Resvalue}. \medskip

\begin{remark}\rm
(1) \ We mentioned already that $\om_{i}=\ell(N_{i}^{\cs_{i}}(f))$. Thus, the computation (and storing) of this invariant saves operations in the {\tt Newton} routine because we know a piori how many coefficients of the $\phi_i$-development must be computed. For the same reason, we compute and store in all types a completely analogous invariant $\alpha_i$ telling a priori the value of $\ell(N_{i}^{\cs_{i}}(g))$.\medskip

(2) \ In step {\bf 9}, the property $\omega_i=1$ implies that $\ty_0$ is already $f$-complete (Definition \ref{optimal}), and $\phi_i^\ty$ is an Okutsu approximation to one of the irreducible factors (say) $F(x)$ of $f(x)$ over $\oo_\p$. In this case, the SFL routine may be applied, and one single iteration of its main loop leads to an Okutsu approximation $\phi$ with $h_\phi\ge 2h_{\phi_i^\ty}$ \cite{GNP}. On the other hand, the standard construction of a representative of the type $\ty$ leads to an Okutsu approximation $\phi$ with $h_\phi\ge h_{\phi_i^\ty}$. Thus, the use of SFL accelerates the process of getting $\phi$   sufficiently close to $F$, to have $(\ty_0;(\phi,\lambda,\psi))\nmid g(x)$ in a future iteration of the WHILE loop.

This acceleration leads to a significant improvement of the routine in cases where the polynomials $f(x)$ and $g(x)$ have $\p$-adic irreducible factors which are very close one to each other.\medskip

(3) \ Step {\bf 10} takes care of the optimization. If $\deg\phi=\deg \phi_i^\ty$, then the future information provided by the Newton polygon $N_{i+1}^-(f)$, with respect to the pair $(\ty,\phi)$, is equivalent to the information provided by  $N_{i}^{\cs_i}(f)$, with respect to the pair $(\ty_0,\phi)$ \cite[Sec. 3.2]{algorithm}. The latter option is more efficient because we work at a lower order.   

\end{remark}

\begin{theorem}\label{algorithm}
The routine {\tt pResultant} terminates and its output value is indeed $v_\p(\res(f,g))$. It requires $O\left(n^{2+\epsilon}+n^{1+\epsilon}(1+\delta)\log q+n^{1+\epsilon}\delta^{2+\epsilon}\right)$
$\p$-small o\-pe\-rations in $A$, where $n=\max\{\deg g,\deg h\}$, $\delta=v_\p(\res(g,h))$ and $q=\#A/\p$.
\end{theorem}

\begin{proof}
Let us call {\tt basic pResultant} the algorithm obtained by eliminating the optimization procedure of step {\bf 10}; that is, by replacing this step by:\medskip

{$\mathbf{10'}$} \ Set $\phi_{i+1}^\ty\leftarrow\phi$, $\cs_{i+1}^\ty\leftarrow0$. 
\medskip

By \cite[Thm. 4.10]{HN}, the {\tt basic pResultant} algorithm terminates and the final output value of the variable {\tt Resvalue} is $v_\p(\res(f,g))$. 

On the other hand, {\tt pResultant} and {\tt basic pResultant} are equivalent algorithms, in the sense that they have the same number of iterations of the WHILE loop and they accumulate the same value to the variable {\tt Resvalue} at each ite\-ration. 
This is a consequence of \cite[Props. 3.4+3.5]{algorithm} and the rest of the arguments of \cite[Secs. 3.2+3.3]{algorithm}, where a completely analogous situation was discussed when we compared basic and optimized versions of the Montes algorithm, both leading to the computation of $\ind_\p(f)$ by the accumulation of partial indices. 

The complexity analysis is obtained by completely analogous arguments to those used in \cite{BNS} to estimate the complexity of the Montes algorithm. 
\end{proof}

\begin{corollary}\label{complexityRes}
If $\p$ is small, then {\tt pResultant} requires $O\left(n^{2+\epsilon}+n^{1+\epsilon}\delta^{2+\epsilon}\right)$ word operations.
\end{corollary}

As it was the case for the routine {\tt pDiscriminant}, if $\chr(K)=0$, the assumption $\res(f,g)\ne0$ may be omitted by considering an upper bound for $v_\p(\res(f,g))$. For instance, for $A=\Z$, $n=\deg f\le m=\deg g$, we have \cite[Thm. 6.23]{vzGG},
$$
\res(f,g)\le \|f\|_2^m\|g\|_2^n=(|a_0|^2+\cdots+|a_n|^2)^{m/2}(|b_0|^2+\cdots+|b_m|^2)^{n/2}.
$$
We leave to the reader the derivation of similar upper bounds, $v_\p(\res(f,g))< N(f,g)$, in the general case. If $\res(f,g)=0$, the WHILE loop of {\tt pResultant} is an endless loop. However, it increases {\tt Resultant} by a positive integer value at each iteration. Thus, we may introduce a control instruction allowing the next iteration as long as the value of {\tt Resultant} is less than $N(f,g)$, and breaking the loop with the output $v_\p(\res(f,g))=\infty$ otherwise.

\section{Numerical tests}
We have implemented the routines {\tt pDiscriminant} and {\tt pResultant} in the case $A=\Z$. They are included in the Magma package {\tt pFactors.m}, which may be downloaded from {\tt http://montesproject.blogspot.com}. 

We present in this section some numerical tests for several polynomials selected from the families of test polynomials in the appendix of \cite{GNP}. All tests have been done in a Linux server, with two Intel Quad Core processors, running at 3.0 Ghz, with 32Gb of RAM memory. Times are expressed in seconds. 

We compare running times of our routines (abbreviated as {\tt pDis}, {\tt pRes} in the tables), with the {\tt naive} routine that computes first $\dsc(g)$ or $\res(g,h)$, and then its $p$-valuation. The numerical results show that the {\tt pDiscriminant} and {\tt pResultant} routines are much more efficient than the corresponding {\tt naive} routines, when the degree of the polynomials and/or the size of the coefficients grow.  

\subsection{Numerical tests of pDiscriminant}\mbox{\null}

\noindent{\bf Example 1.}
Let $p$ be a prime number and $n$ a positive integer. Consider the polynomial of degree $n$:
$$
f(x)=\left(x+p+p^2+\cdots+p^{20}\right)^n+p^{20n+1}.
$$
This polynomial is irreducible over $\Z_p[x]$ and it has Okutsu depth equal to one. It has $v_p(\dsc(f))=(20n+1)(n-1)+nv_p(n)$ \cite[App.]{GNP}.\bigskip

\begin{center}
\begin{tabular}{|c|c|c|c|r|}
\hline
 $p$&$\deg f$&$v_p(\dsc(f))$&{\tt pDisc}&{\tt naive}\\\hline
$2$&$20$&$7659$&$0.01$&$0.00$\\\hline
$2$&$50$&$49099$&$0.02$&$0.83$\\\hline
$2$&$100$&$198299$&$0.07$&$26.60$\\\hline
$2$&$150$&$447299$&$0.13$&$190.57$\\\hline
$59$&$20$&$7619$&$0.02$&$0.00$\\\hline
$59$&$50$&$49049$&$0.06$&$66.14$\\\hline
$59$&$100$&$198099$&$0.22$&$1887.55$\\\hline
$59$&$150$&$447149$&$0.59$&$12225.08$\\\hline
$59$&$500$&$4990499$&$12.80$&$>24$\,hours
\\\hline
\end{tabular}
\end{center}\bigskip

\noindent{\bf Example 2.}
Let us test the influence of the existence of many $p$-adic irreducible factors.
Let $p>5$ be a prime number, and $m$ a positive integer, $m<p/2$. Let $g_0(x)=x^ {10}+2p^{11}$. Consider the polynomial of degree $10m$:
$$
g(x)=g_0(x)g_0(x+2)\cdots g_0(x+2(m-1))+2p^{110m}.
$$
This polynomial has $m$ irreducible factors of degree $10$ over $\Z_p[x]$, all of them with Okutsu depth equal to one. It has $v_p(\dsc(g))=99m$ \cite[App.]{GNP}.\bigskip

\begin{center}
\begin{tabular}{|c|c|c|c|c|r|}
\hline
 $p$&$m$&$\deg g$&$v_p(\dsc(g))$&{\tt pDisc}&{\tt naive}\\\hline
$7$&$3$&$30$&$297$&$0.01$&$0.22$\\\hline
$23$&$5$&$50$&$495$&$0.03$&$8.30$\\\hline
$23$&$10$&$100$&$990$&$0.04$&$243.76$\\\hline
$43$&$5$&$50$&$495$&$0.04$&$12.92$\\\hline
$43$&$10$&$100$&$990$&$0.04$&$378.17$\\\hline
$43$&$15$&$150$&$1485$&$0.04$&$2598.82$\\\hline
$101$&$50$&$500$&$4950$&$0.27$&$>24$\,hours\\\hline
\end{tabular}
\end{center}\bigskip

\noindent{\bf Example 3.}
Finally, let us test the influence of the growth of the Okutsu depth in both routines. Let $p>3$ be a prime number. Consider the following polynomials:
\begin{center}
$$\as{1.2}
\begin{array}{l}
\displaystyle E_1(x)=x^2+p\\ 
\displaystyle E_2(x)=E_1(x)^{2}+(p-1)p^{3}x\\
\displaystyle E_3(x)=E_2(x)^{3}+p^{11}\\
\displaystyle E_4(x)=E_3(x)^{3}+p^{29}xE_2(x)\\
\displaystyle E_5(x)=E_4(x)^{2}+(p-1)p^{42}xE_1(x)E_3(x)^2 \\ 
\displaystyle E_6(x)=E_5(x)^{2}+p^{88}xE_3(x)E_4(x)\\
\displaystyle E_7(x)=E_6(x)^{3}+p^{295}E_2(x)E_4(x)E_5(x)\\
\displaystyle E_8(x)=E_7(x)^{2}+(p-1)p^{632}xE_1(x)E_2(x)^2E_3(x)^2E_6(x)\\
\end{array}
$$
\end{center}

These polynomials are all irreducible over $\Z_p[x]$ and the Okutsu depth of $E_j$ is equal to $j$, for all $1\le j\le 8$ \cite[App.]{GNP}.\bigskip

\begin{center}
\begin{tabular}{|c|c|c|c|c|r|}
\hline
$p$&$j$&$\deg E_j$&$v_p(\dsc(E_j))$&{\tt pDisc}&{\tt naive}\\\hline
$5$&$5$&$72$&$4671$&$0.01$&$0.09$\\\hline
$5$&$6$&$144$&$18899$&$0.02$&$1.66$\\\hline
$5$&$7$&$432$&$171383$&$0.15$&$204.54$\\\hline
$5$&$8$&$864$&$686825$&$0.72$&$4565.77$\\\hline
$61$&$5$&$72$&$4671$&$0.01$&$0.39$\\\hline
$61$&$6$&$144$&$18899$&$0.02$&$9.63$\\\hline
$61$&$7$&$432$&$171383$&$0.22$&$1580.66$\\\hline
$61$&$8$&$864$&$686825$&$1.25$&$35623.63$
\\\hline
\end{tabular}
\end{center}\bigskip

\subsection{Numerical tests of pResultant}\mbox{\null}

\noindent{\bf Example 4.}
Let $p>5$ be a prime number, $m$ a positive integer, $m<p/2$, and $g(x)$ the polynomial of Example 2. Let $f(x)$ be the polynomial of Example 1 of degree $n=10m=\deg g$.\medskip

\begin{center}
\begin{tabular}{|c|c|c|r|r|}
\hline
 $p$&$\deg f=\deg g$&$v_p(\res(f,g))$&{\tt pRes}&{\tt naive}\\\hline
$7$&$30$&$300$&$0.01$&$0.94$\\\hline
$11$&$50$&$500$&$0.00$&$19.58$\\\hline
$23$&$100$&$1000$&$0.01$&$1060.03$\\\hline
$31$&$150$&$1500$&$0.01$&$8433.39$\\\hline
$43$&$200$&$2000$&$0.03$&$38430.25$\\\hline
$101$&$500$&$5000$&$0.27$&$>24$\,hours\\\hline
\end{tabular}
\end{center}\bigskip

\noindent{\bf Example 5.}
Let $p>3$ be a prime number, and consider the polynomials $E_j(x)$, $1\le j\le 8$, of Example 3.\medskip

\begin{center}
\begin{tabular}{|c|c|c|c|r|r|}
\hline
$p$&$i$&$j$&$v_p(\res(E_i,E_j))$&{\tt pRes}&{\tt naive}\\\hline
$5$&$5$&$6$&$9557$&$0.04$&$0.23$\\\hline
$5$&$5$&$7$&$28671$&$0.14$&$2.70$\\\hline
$5$&$5$&$8$&$57342$&$0.61$&$13.76$\\\hline
$5$&$6$&$7$&$57343$&$0.15$&$13.11$\\\hline
$5$&$6$&$8$&$114686$&$0.67$&$66.12$\\\hline
$5$&$7$&$8$&$344059$&$0.83$&$885.89$\\\hline
$101$&$5$&$6$&$9557$&$0.05$&$1.95$\\\hline
$101$&$5$&$7$&$28671$&$0.22$&$28.11$\\\hline
$101$&$5$&$8$&$57342$&$1.12$&$150.30$\\\hline
$101$&$6$&$7$&$57343$&$0.25$&$141.52$\\\hline
$101$&$6$&$8$&$114686$&$1.31$&$746.78$\\\hline
$101$&$7$&$8$&$344059$&$1.60$&$9335.33$\\\hline
\end{tabular}
\end{center}


\bigskip



\begin{thebibliography}{}
\bibitem{BNS}
J.-D. Bauch, E. Nart, H. D. Stainsby, \emph{Complexity of OM factorizations of polynomials over local fields},	arXiv:1204.4671v1 [math.NT].

\bibitem{FV} D. Ford, O. Veres, \emph{On the Complexity of the Montes Ideal Factorization Algorithm}, in G. Hanrot and F. Morain and E. {Thom\'e},
\emph{Algorithmic Number Theory},
\emph{9th International Symposium, ANTS-IX, Nancy, France, July 19-23, 2010},
LNCS, Springer Verlag 2010.


\bibitem{vzGG}
J. von zur Gathen, J. Gerhard, \emph{Modern Computer Algebra}, second Edition, Cambridge University Press, 2003.

\bibitem{okutsu}
J. Gu\`{a}rdia, J.  Montes, E.  Nart, \emph{Okutsu invariants and Newton polygons}, Acta Arithmetica {\bf145} (2010), 83--108.

\bibitem{algorithm}
J. Gu\`{a}rdia, J.  Montes, E.  Nart, \emph{Higher  Newton polygons in the computation of discriminants and prime ideal decomposition in number fields},
Journal de Th\'eorie des Nombres de Bordeaux {\bf 23} (2011), no. 3, 667--696.

\bibitem{HN}
J. Gu\`{a}rdia, J. Montes, E. Nart, \emph{Newton polygons of higher order in algebraic number theory}, Transactions of the American Mathematical Society  {\bf 364} (2012), no. 1, 361--416.

\bibitem{newapp}
J. Gu\`{a}rdia, J. Montes, E.  Nart, \emph{A new computational approach to ideal theory in number fields},  arXiv:1005.1156v3[math.NT].

\bibitem{GNP}J. Gu\`{a}rdia, E.  Nart, S. Pauli, \emph{Single-factor lifting and factorization of polynomials over local fields}, Journal of Symbolic Computation (2012), doi:10.1016/j.jsc.2012.03.001.

\bibitem{mahler} K. Mahler, \emph{An inequality for the discriminant of a polynomial}, Michigan Math. J. {\bf11} (1964), 257--262.

\bibitem{survey} E. Nart, \emph{Okutsu-Montes representations of prime ideals of one-dimensional integral closures}, Publicacions Matem\`atiques {\bf55} (2011), no. 3, 261--294.

\bibitem{Ok}
K. Okutsu, \emph{Construction of integral basis, I, II}, Proceedings of the Japan Academy {\bf 58}, Ser. A (1982), 47--49, 87--89.

\bibitem{ore} \O{}. Ore, \emph{Bestimmung der Diskriminanten algebraischer K\"orper}, Acta Mathematica {\bf45}(1925), pp. 303--344.

\bibitem{pauli}
S. Pauli, \emph{Factoring polynomials over local fields, II}, in G. Hanrot and F. Morain and E. {Thom\'e},
\emph{Algorithmic Number Theory},
\emph{9th International Symposium, ANTS-IX, Nancy, France, July 19-23, 2010},
LNCS, Springer Verlag 2010.

\bibitem{SS}
A. Sch\"onhage, V. Strassen,
\emph{Schnelle Multiplikation gro$\beta$er Zahlen},
Computing, \textbf{7} (1971), 281--292

\bibitem{serre} J. P. Serre, \emph{Corps Locaux}, second Edition, Hermann, Paris, 1968.

\end{thebibliography}
\end{document}